\documentclass{article}
\usepackage{amsmath,amssymb}
\usepackage{amscd}
\usepackage{comment}
\usepackage{color}
\usepackage{enumerate}
\usepackage{cases}
\usepackage{mathtools}
\usepackage[all]{xy}
\usepackage{amsthm}
\usepackage{pb-diagram}
\theoremstyle{definition}
\newtheorem{defi}{Definition}%[section]
\newtheorem{theo}[defi]{Theorem}
\newtheorem{lemm}[defi]{Lemma}

\newtheorem{ex}[defi]{Example}
\newtheorem{rem}[defi]{Remark}

\def\det{{\rm det}}

\def\End{{\rm End}}

\def\det{{\rm det}}

\def\R{{\mathbb R}}
\def\Z{{\mathbb Z}}
\def\C{{\mathbb C}}
\def\N{{\mathbb N}}

\def\inum{{\sqrt{-1}}}

\begin{document}

\title {Generalizations of Hermitian-Einstein equation of cyclic Higgs bundles, their heat equation, and inequality estimates}
\author {Natsuo Miyatake}
\date{}
\maketitle
\begin{abstract} 
We introduce some generalizations of the Hermitian-Einstein equation for diagonal harmonic metrics on cyclic Higgs bundles, including a generalization using subharmonic functions. When the coefficients are all smooth, we prove the existence, uniqueness, and convergence of the solution of their heat equations with Dirichlet boundary conditions. We also generalize two inequality estimates for solutions of the Hermitian-Einstein equation of cyclic Higgs bundles.
\end{abstract}
\section{Introduction}
Let $V$ be a real vector space defined as $V\coloneqq \{x=(x_1,\dots, x_r)\in\R^r\mid x_1+\cdots+x_r=0\}$. We take a generator $v_1,\dots, v_r$ of $V$ defined as $v_j\coloneqq u_{j+1}-u_j \ (j=1,\dots, r-1), \ v_r\coloneqq u_1-u_r$, where we denote by $u_1,\dots, u_r$ the canonical basis of $\R^r$. Let $(M,g_M)$ be a Riemannian manifold with geometric Laplacian $\Delta_{g_M}$. Let $a_1,\dots, a_r:M \rightarrow \R_{\geq0}$ be non-negative functions and $w:M\rightarrow V$ a $V$-valued function. We suppose that for each $j=1,\dots, r$, the function $a_j$ is not identically zero. We consider the following PDE on $M$:
\begin{align}
\Delta_{g_M}\xi+\sum_{j=1}^ra_je^{(v_j, \xi)}v_j=w. \label{KW}
\end{align}
Our main theorems are as follows:
\begin{theo}\label{main theorem 1} {\it Let $(M,g_M)$ be a compact Riemannian manifold with a smooth, possibly empty boundary $\partial M$. Suppose that $a_1,\dots, a_r,w$ are all smooth functions. Then, for any smooth function $\eta: M\rightarrow V$, there exists a unique one parameter family $(\xi_t)_{t\in[0,\infty)}$ that is continuous on $M\times [0,\infty)$ and smooth on $M\times (0,\infty)$, such that: 
\begin{align}
&\frac{d\xi_t}{dt}+\Delta_{g_M}\xi_t+\sum_{j=1}^ra_je^{(v_j,\xi_t)}v_j-w=0 \ \text{for all $t\in(0,\infty)$}, \label{heat equation}\\
&\xi_t\left.\right|_{\partial M}=\eta\left.\right|_{\partial M} \ \text{for all $t\in[0,\infty)$}, \ \xi_0=\eta. \label{boundary}
\end{align}
Moreover, for the solution $(\xi_t)_{t\in [0,\infty)}$ of (\ref{heat equation}), with the boundary condition (\ref{boundary}), the following holds:
\begin{enumerate}[(i)]
\item Suppose that the boundary $\partial M$ is not empty. Then there exists a sequence $(t_i)_{i\in \N}$ going to infinity such that $(\xi_{t_i})_{i\in \N}$ converges in $C^\infty$ to a smooth solution $\xi$ of equation (\ref{KW}) satisfying the Dirichlet boudary condition $\xi\left.\right|_{\partial M}=\eta\left.\right|_{\partial M}$. Furthermore, the solutions of equation (\ref{KW}) satisfying the same Dirichlet boundary condition are unique.
\item Suppose that the boundary $\partial M$ is empty and that for each $j=1,\dots, r$, $a_j^{-1}(0)$ is a measure $0$ set and $\log a_j$ is integranle. Then there exists a sequence $(t_i)_{i\in \N}$ going to infinity such that $\xi_{t_i}$ converges in $C^\infty$ to a smooth solution $\xi$ 
of equation (\ref{KW}). Furthermore, the solution of equation (\ref{KW}) is unique.
\end{enumerate}
}
\end{theo}
\begin{theo}\label{main theorem 3}{\it Let $\xi,\xi^\prime:M\rightarrow V$ be $V$-valued $C^2$-functions. Then the following holds:
\begin{align*}
&\Delta_{g_M}\log|\sum_{j=1}^re^{(\xi-\xi^\prime,u_j)}| \\
\leq &|\Delta_{g_M}\xi+\sum_{j=1}^ra_je^{(v_j,\xi)}v_j-w|+|\Delta_{g_M}\xi^\prime+\sum_{j=1}^ra_je^{(v_j,\xi^\prime)}v_j-w|.
\end{align*}
}
\end{theo}
\begin{theo}\label{main theorem 4}{\it For each $j=1,\dots, r$, suppose that $a_j$ is a $C^2$-function, and that the following holds on $M\backslash a_j^{-1}(0)$.
\begin{align*}
\Delta_{g_M}\log (a_j)\leq -(v_j,w).
\end{align*}
Let $\xi:M\rightarrow V$ be a $C^2$-solution of (\ref{KW}). Then the following holds on $M\backslash\bigcup_{j=1}^ra_j^{-1}(0)$:
\begin{align*}
\Delta_{g_M}\log\bigl(\sum_{j=1}^ra_je^{(v_j,\xi)}\bigr)\leq -\frac{\left|\sum_{j=1}^ra_je^{(v_j,\xi)}v_j\right|^2}{\left|\sum_{j=1}^ra_je^{(v_j,\xi)}\right|}.
\end{align*}
}
\end{theo}
\begin{rem} A function on a manifold with boundary is said to be {\it smooth} if it is smooth on $M\backslash \partial M$, and if it has a smooth extension at each point on the boundary.
\end{rem}
\begin{rem} All of the above theorems can easily be generalized to theorems for the more generalized equation investigated in \cite{Miy1, Miy3}.
\end{rem}
Equation (\ref{KW}) is a generalization of the Hermitian-Einstein equation for diagonal harmonic metrics on cyclic Higgs bundles \cite{Bar1, DL1}. For cyclic Higgs bundles, Theorem \ref{main theorem 1} follows from \cite{Don1, Sim1}, Theorem \ref{main theorem 3} follows from \cite[Lemma 3.1]{Sim1}, and Theorem \ref{main theorem 4} follows from \cite[proof of Lemma 10.1]{Sim1} (see also \cite{LM1, LM2}). These are very fundamental, and important theorems for Higgs bundles. We briefly recall the definition of cyclic Higgs bundles and their Hermitian-Einstein equation. Let $X$ be a connected Riemann surface with the canonical bundle $K_X\rightarrow X$. 
We take a square root $K_X^{1/2}\rightarrow X$ and we define a holomorphic vector bundle $E\rightarrow X$ of rank $r$ as $E\coloneqq K_X^{(r-1)/2}\oplus K_X^{(r-3)/2}\oplus\cdots\oplus K_X^{-(r-3)/2}\oplus K_X^{-(r-1)/2}$. We take a $q\in H^0(K_X^r)$. We set $\Phi(q)_{j+1,j}$ to equal 1 for each $j=1,\dots, r-1$, and $\Phi(q)_{1,r}$ to equal $q$. We define $\Phi(q)\in H^0(\End E\otimes K_X)$ as $\Phi(q)\coloneqq \sum_{j=1}^{r-1}\Phi(q)_{j+1,j}+\Phi(q)_{1,r}$, where $\Phi(q)_{i,j}$ is considered to be the $(i,j)$-component of $\Phi(q)$, and $1$ (resp. $q$) is considered to be a $K_X^{-1}$ (resp. $K_X^{r-1}$)-valued holomorphic 1-form. We call $(E,\Phi(q))$ a cyclic Higgs bundle (see \cite{DL1} for a more generalization). We take a diagonal Hermitian metric $h=(h_1,\dots, h_r)$ on $E$ with curvature $F_h$ such that $\det(h)=1$. We also take a K\"ahler form $\omega_X$. We denote by $\Lambda_{\omega_X}$ the dual of $\omega_X\wedge$, and by $\Delta_{\omega_X}$ the geometric Laplacian. 
The following PDE for a $V$-valued function $\xi=(f_1,\dots, f_r): X\rightarrow V$ is the Hermitian-Einstein equation of the metric $(e^{f_1}h_1,\dots, e^{f_r}h_r)$ on $(E,\Phi(q))$:
\begin{align}
\Delta_{\omega_X}\xi+\sum_{j=1}^r4k_je^{(v_j,\xi)}v_j=-2\inum \Lambda_{\omega_X}F_{h}, \label{HEeq}
\end{align}
where $\R^r$ is identified with diagonal matrices, and we denote by $k_1,\dots, k_r$ the positive functions defined by $k_j\coloneqq |1|_{h,\omega_X}^2 \ (j=1,\dots, r-1), \ k_r\coloneqq |q|_{h,\omega_X}^2$. Equation (\ref{HEeq}) is also called Toda lattice with opposite sign (see \cite{GL1}). Equation (\ref{KW}) is a generalization of (\ref{HEeq}). Equation (\ref{KW}) also includes the following examples: 

\begin{ex}[Generalization by using subharmonic functions]\label{ex3} We consider the case where $X$ is a domain of $\C$ for simplicity. Let $\omega_X$ be the restriction of the standard K\"ahler form on $\C$ to $X$, and $h=(h_1,\dots, h_r)$ the metric on $E$ induced by $\omega_X$. Let $f:X\rightarrow \C$ be a holomorphic function and we set $q\coloneqq f(dz)^r$. We normalize the metric $h$ so that $|q|_{h,\omega_X}^2=|f|^2$. We consider a generalization of (\ref{HEeq}) obtained by replacing $|f|^2=e^{\log|f|^2}$ by $e^{\eta}$ with an arbitrary subharmonic function $\eta: X\rightarrow \R$ (cf. \cite{Ran1}):
\begin{align}
\Delta_{\omega_X}\xi+\sum_{j=1}^r4k_j^\prime e^{(v_j,\xi)}v_j=0, \label{HEeq2}
\end{align}
where $k_j^\prime$ $(j=1,\dots, r)$ are defined as $k_j^\prime\coloneqq |1|_{h,\omega_X}^2$ $(j=1,\dots, r-1)$,  $k_r^\prime\coloneqq e^\eta$. Consider the case that $\eta=\frac{1}{N}\log|f|^2$ for some $N\in\Z_{>0}$ and some holomorphic function $f:X\rightarrow \C$. Then on a subdomain $D\subseteq X$ such that we can choose a well-defined single-valued function $f^{1/N}$, equation (\ref{HEeq2}) is the Hermitian-Einstein equation for a cyclic Higgs bundle associated with $q=f^{\frac{1}{N}}(dz)^r$, and a solution of equation (\ref{HEeq2}) gives a harmonic bundle on $D$. Let $(\eta_j=\frac{1}{N_j}\log|f_j|^2)_{j\in\N}$ be a sequence of subharmonic functions with $N_j\in\Z_{>0}$ and holomorphic functions $f_j:X\rightarrow \C$ $(j=1,2, \dots, )$. If $\eta$ is obtained as a limit of a sequence $(\eta_j=\frac{1}{N_j}\log|f_j|^2)_{j\in\N}$ with respect to some topology on the space of subharmonic functions, then equation (\ref{HEeq2}) associated with $\eta$ can be considered to be a limit of equations (\ref{HEeq2}) associated with $\eta_j=\frac{1}{N}\log|f_j|^2$ $(j=1,2,\dots, )$. The above sentence ``equation (\ref{HEeq2}) associated with $\eta$ can be considered to be a limit of equations (\ref{HEeq2}) associated with $\eta_j=\frac{1}{N}\log|f_j|^2$ $(j=1,2,\dots, )$'' is not a rigorous statement, but the author hopes that this statement will be a mathematically rigorous claim.

%If $\eta$ is obtained as a limit of a sequence $(\eta_j=\frac{1}{N_j}\log|f_j|)_{j\in\N}$ with $N_j\in\Z_{>0}$ and holomorphic functions $f_j:X\rightarrow \C$ $(j=1,2, \dots, )$, then equation (\ref{HEeq2}) can be considered to be a limit of equation (\ref{HEeq2}) associated with $\eta_j=\frac{1}{N}\log|f_j|^2$ $(j=1,2,\dots, )$. 
%the author hopes that equation (\ref{HEeq2}) can in some sense be regarded as a limit of the sequence of Hermite-Einstein equations.
\end{ex}
\begin{ex} [Perturbation by currents]\label{ex1} Let $T=(T_1,\dots, T_r)$ be a $V$-valued real current on $X$. We consider equation (\ref{HEeq}) perturbed by $T$:
\begin{align}
\Delta_{\omega_X}\xi+\sum_{j=1}^r4k_je^{(v_j,\xi)}v_j=-2\inum \Lambda_{\omega_X}F_{h_j}+T_j\label{HEeqT}
\end{align}
On an open subset $U\subseteq X$ such that $T\left.\right|_U=0$, (\ref{HEeqT}) is the usual Hermitian-Einstein equation, and we obtain a harmonic bundle on $U$ from a solution of (\ref{HEeqT}). 
\end{ex}

\begin{ex}[Cyclic Higgs bundles with non-holomorphic Higgs fields]\label{ex2} Let $q$ be a non-holomorphic section of $K_X^r\rightarrow X$. We consider equation (\ref{HEeq}) for a non-holomorphic $q$. On an open subset $U\subseteq X$ such that $\bar{\partial}q=0$, a solution of the Hermitian-Einstein equation gives a harmonic bundle (see also \cite{Miy4}). 
\end{ex}

We can also consider the combination of the above examples. The above generalization can also be considered for the more generalized cyclic Higgs bundles \cite{DL1}, and for cyclic Higgs bundles on real 3-dimensional manifolds with transverse complex structures (see \cite{Miy3}). 

%On compact Riemannian manifolds, the following holds:

%\begin{theo}[\cite{Miy1, Miy3}] \label{not main theorem}{\it Suppose that $M$ is compact. Suppose also that $a_1,\dots, a_r, w$ are smooth and that for each $j=1,\dots, r$, $a_j^{-1}(0)$ is a measure zero and $\log a_j$ is integrable. Then a smooth solution of (\ref{KW}) uniquely exists. Moreover, if there exists a Riemannian foliation on $(M,g_M)$ and if $a_1,\dots, a_r, w$ are all basic functions concerning the foliation, then the unique solution of (\ref{KW}) is a basic function.} \end{theo}

\section{Proof}
Proof of Theorem \ref{main theorem 1} heavily relies on the techniques used in \cite{ Don0, Don1, Ham1, Sim1}. We prepare some lemmas. 
\begin{lemm}\label{lemma1}{\it Let $(\xi_t)_{t\in I}$ and $(\xi_t^\prime)_{t\in I}$ be solutions of (\ref{heat equation}) defined on an interval $I$. Then the following holds:
\begin{align*}
\frac{1}{2}\left(\frac{d}{dt}+\Delta_{g_M}\right)|\xi_t-\xi_t^\prime|^2\leq -\sum_{j=1}^ra_j(e^{(v_j,\xi_t)}v_j-e^{(v_j,\xi_t^\prime)}v_j,\xi_t-\xi_t^\prime)\leq 0.
\end{align*}
}
\end{lemm}
\begin{proof} By calculation, we have
\begin{align*}
&\frac{1}{2}\left(\frac{d}{dt}+\Delta_{g_M}\right)|\xi_t-\xi_t^\prime|^2 \\
=&\left(\left\{\frac{d}{dt}+\Delta_{g_M}\right\}(\xi_t-\xi_t^\prime), \xi_t-\xi_t^\prime\right) -|d(\xi_t-\xi_t^\prime)|^2 \\
\leq & -\sum_{j=1}^ra_j(e^{(v_j,\xi_t)}v_j-e^{(v_j,\xi_t^\prime)}v_j,\xi_t-\xi_t^\prime) \\
\leq & 0
\end{align*}
Then we have the result.
\end{proof}
As a special case of Lemma \ref{lemma1}, we have
\begin{lemm}\label{xi-}{\it Let $\xi$ and $\xi^\prime$ be solutions of (\ref{KW}). Then the following holds:
\begin{align*}
\frac{1}{2}\Delta_{g_M}|\xi-\xi^\prime|^2\leq -\sum_{j=1}^ra_j(e^{(v_j,\xi)}v_j-e^{(v_j,\xi^\prime)}v_j,\xi-\xi^\prime)\leq 0.
\end{align*}
}
\end{lemm}
\begin{rem}
For the proof of Theorem \ref{main theorem 1}, for $V$-valued functions $\xi$ and $\xi^\prime$, we use the above function $|\xi-\xi^\prime|^2$ instead of the following function used in \cite{Don0, Don1, Sim1}:
\begin{align*}
\sigma(\xi,\xi^\prime)\coloneqq \sum_{j=1}^r(e^{(\xi-\xi^\prime,u_j)}+e^{(\xi^\prime-\xi,u_j)})-2r.
\end{align*} 
Note, however, that all arguments in the proof of Theorem \ref{main theorem 1} can be done just as well using $\sigma(\xi,\xi^\prime)$ instead of $|\xi-\xi^\prime|^2$.
\end{rem}

\begin{defi}\label{F, G}
For a $V$-valued function $\xi$, we set 
\begin{align*}
&F(\xi)\coloneqq \Delta_{g_M}\xi+\sum_{j=1}^ra_je^{(v_j,\xi)}v_j-w.
%&G(\xi)\coloneqq F(\xi)-\Delta_{g_M}\xi-w.
\end{align*}
\end{defi}
The following holds:
\begin{comment}
\begin{lemm}{\it Let $(\xi_t)_{t\in I}$ be a solution of (\ref{heat equation}) defined on an interval $I$. Then for each $k\in\Z_{\geq 0}$, the following holds:
\begin{align*}
\frac{1}{2}\left(\frac{d}{dt}+\Delta_{g_M}\right)|\Delta_{g_M}^k\xi_t|^2\leq -(\Delta_{g_M}^kG(\xi), \Delta_{g_M}^k\xi),
\end{align*}
where we denote by $\Delta_{g_M}^k$ the Laplacian multiplied $k$-times: $\Delta_{g_M}\circ\cdots\circ \Delta_{g_M}$.
}
\end{lemm}
\begin{proof}
By calculation, we have
\begin{align*}
\frac{1}{2}\left(\frac{d}{dt}+\Delta_{g_M}\right)\left|\Delta_{g_M}^k\xi_t\right|^2&=\left(\left(\frac{d}{dt}+\Delta_{g_M}\right)\Delta_{g_M}^k\xi_t,\Delta_{g_M}^k\xi_t\right) -\left|d\Delta_{g_M}^k\xi_t\right|^2 \\
&\leq -(\Delta_{g_M}^kG(\xi), \Delta_{g_M}^k\xi).
\end{align*}
Then we have the claim.
\end{proof}
\end{comment}
\begin{lemm}\label{lemma2} {\it Let $(\xi_t)_{t\in I}$ be a solution of (\ref{heat equation}) defined on an interval $I$. Then the following holds:
\begin{align*}
\frac{1}{2}\left(\frac{d}{dt}+\Delta_{g_M}\right)|F(\xi_t)|^2\leq -\sum_{j=1}^ra_je^{(v_j,\xi_t)}(v_j,F(\xi_t))^2\leq 0.
\end{align*}
}
\end{lemm}
\begin{proof}
The following holds:
\begin{align*}
&\frac{1}{2}\left(\frac{d}{dt}+\Delta_{g_M}\right)|F(\xi_t)|^2 \\
=&\left(\left\{\frac{d}{dt}+\Delta_{g_M}\right\}F(\xi_t), F(\xi_t)\right) -|dF(\xi_t)|^2 \\
\leq &\left(\left\{\frac{d}{dt}+\Delta_{g_M}\right\}F(\xi_t), F(\xi_t)\right).
\end{align*}
By calculation, we have
\begin{align*}
\left(\frac{d}{dt}+\Delta_{g_M}\right)F(\xi_t) &=\Delta_{g_M}\frac{d\xi_t}{dt}+\frac{d}{dt}\sum_{j=1}^ra_je^{(v_j,\xi_t)}v_j+\Delta_{g_M}F(\xi_t) \\
&=-\sum_{j=1}^ra_je^{(v_j,\xi_t)}(v_j,F(\xi_t))v_j.
\end{align*}
Then we have the claim.
\end{proof}
For each $j=0, 1,2,\dots,$ let $W_j\coloneqq \bigotimes^j T^\ast M$. We denote by $\Delta_{p,W_j}:\Omega^p(W_j)\rightarrow \Omega^p(W_j)$ the Laplacian acting on the space of $p$-forms which takes values in $W_j$. Let $R_j\coloneqq \Delta_{0, W_j}-\Delta_{1,W_{j-1}}:\Omega^0(W_j)\rightarrow \Omega^0(W_j)$, where $\Omega^0(W_j)$ is identified with $\Omega^1(W_{j-1})$ through the natural identification. From the Weitzenb\"ock formula, it follows that $R_j$ is $C^\infty(M,\R)$-linear, and thus $R_j\in\Gamma(\End(W_j))$. The following holds:
\begin{lemm}\label{curvature}{\it Let $(\xi_t)_{t\in I}$ be a solution of (\ref{heat equation}) defined on an interval $I$. Then for each $k=0,1,2, \dots$, the following holds:
\begin{align*}
&\frac{1}{2}\left(\frac{d}{dt}+\Delta_{g_M}\right)|\nabla^kF(\xi_t)|^2 \\
\leq &\sum_{j=0}^{k-1}(\nabla^j(R_{k-j}\nabla^{k-j}F(\xi_t)), \nabla^kF(\xi_t))-\sum_{j=1}^r(\nabla^k\{a_je^{(v_j,\xi_t)}(v_j,F(\xi_t))\})(v_j,\nabla^k F(\xi_t)),
\end{align*}
where $\nabla^k$ denotes $\nabla\circ\cdots\circ \nabla:\Gamma(W_0)\rightarrow \Gamma(W_k)$.
}
\end{lemm}
\begin{proof}
By calculation, we have
\begin{align*}
&\frac{1}{2}\left(\frac{d}{dt}+\Delta_{g_M}\right)|\nabla^kF(\xi_t)|^2 \\
=&\left(\left\{\frac{d}{dt}+\Delta_{0,W_k}\right\}\nabla^kF(\xi_t),\nabla^kF(\xi_t)\right)-|\nabla(\nabla^kF(\xi_t))|^2 \\
\leq &\left(\left\{\frac{d}{dt}+\Delta_{0,W_k}\right\}\nabla^kF(\xi_t),\nabla^kF(\xi_t)\right).
\end{align*}
By induction on $k$, we have
\begin{align*}
\Delta_{0, W_k}\nabla^k F(\xi_t)=\sum_{j=0}^{k-1}\nabla^j \{R_{k-j}\nabla^{k-j}F(\xi_t)\}+\nabla^k(\Delta_{g_M}F(\xi_t)),
\end{align*}
where we have used $\nabla\circ \Delta_{0, W_j}=\Delta_{1,W_j}\circ \nabla$. Therefore we have
\begin{align*}
&\left(\left\{\frac{d}{dt}+\Delta_{0,W_k}\right\}\nabla^kF(\xi_t),\nabla^kF(\xi_t)\right) \\
=&\sum_{j=0}^{k-1}(\nabla^j(R_{k-j}\nabla^{k-j}F(\xi_t)), \nabla^kF(\xi_t))+\left(\nabla^k\left(\frac{d}{dt}+\Delta_{g_M}\right)F(\xi_t), \nabla^k F(\xi_t)\right).
\end{align*}
From the same calculation as in the proof of Lemma \ref{lemma2}, we have the following:
\begin{align*}
\left(\frac{d}{dt}+\Delta_{g_M}\right)F(\xi_t) 
&=-\sum_{j=1}^ra_je^{(v_j,\xi_t)}(v_j,F(\xi_t))v_j.
\end{align*}
From this equation the desired claim is deduced.
\end{proof}

\begin{lemm}[\cite{Ham1,Kob1}] \label{max}{\it Let $f:[a,b]\times M\rightarrow \R$ be a continuous function such that $f\left.\right|_{[a,b]\times\partial M}$ is constant and that $f$ is $C^2$ on $(a,b ]\times M\backslash \partial M$. Suppose that the following holds on $(a,b]\times M\backslash \partial M$:
\begin{align*}
\left(\frac{d}{dt}+\Delta_{g_M}\right)f\leq 0.
\end{align*}
Then $\sup_{p\in M}f(t,p)$ is a monotone decreasing function for $t\in[a,b]$.
} 
\end{lemm}
\begin{proof} Following \cite{Kob1}, for the readers convenience, we give a proof of the above lemma. For each $\epsilon>0$, let $f_\epsilon\coloneqq f-\epsilon t$. It is enough to show that for each $t_1<t_2$, $\sup_{p\in M}f_\epsilon(t_1,p)\geq \sup_{p\in M}f_\epsilon(t_2,p)$ holds for all $\epsilon>0$. Without loss of generality, we can assume that $t_1=a$. Let $(\bar{t}, \bar{p})\in [a,b]\times M$ be a point such that $f_\epsilon(\bar{t},\bar{p})=\sup_{(t,p)\in[a,b]\times M}f_\epsilon(t,p)$. We show that $\bar{t}=a$. We first consider the case that $\bar{p}\in\partial M$. Then $f_\epsilon(\bar{t},\bar{p})=-\epsilon \bar{t}+C$ with some constant $C$. Therefore, $\bar{t}=a$. We next consider the case that $\bar{p}\notin \partial M$. Suppose that $\bar{t}\in (a,b]$. Then the following holds:
\begin{align*}
\frac{df_\epsilon}{dt}(\bar{t},\bar{p})\leq (-\Delta_{g_M}f)(\bar{t},\bar{p})-\epsilon\leq -\epsilon <0.
\end{align*}
This contradicts the assumption that $f_\epsilon(\bar{t},\bar{p})=\sup_{(t,p)\in[a,b]\times M}f_\epsilon(t,p)$. Therefore, it must be that $\bar{t}=a$ and thus we have the desired claim.
\end{proof}
\begin{proof}[Proof of Theorem \ref{main theorem 1}] 
We first show the uniqueness of the solution of equation (\ref{heat equation}). Let $(\xi_t)_{t\in[0,\infty)}$ and $(\xi^\prime_t)_{t\in[0,\infty)}$ be solutions of equation (\ref{heat equation}) such that $\xi_0=\xi^\prime_0$ and that $\xi_t\left.\right|_{\partial M}=\xi^\prime_t\left.\right|_{\partial M}$ for all $t\in[0,\infty)$. We define a function $f:[0,\infty)\times M\rightarrow \R$ as $f\coloneqq |\xi_t-\xi_t^\prime|^2$. Then, the function $f$ vanishes on $[0,\infty)\times \partial M\cup\{0\}\times M$. Therefore from Lemma \ref{lemma1} and Lemma \ref{max}, $f$ vanishes on $[0,\infty)\times M$ and thus $\xi_t=\xi^\prime_t$ for all $t\in [0,\infty)$. 

Secondly, we demonstrate the existence of a solution for a short time interval for equation (\ref{heat equation}). For each $T>0$, let $C^\infty([0,T]/0\times M,V)$ be the space of $V$-valued $C^\infty$-functions on $[0,T]\times M$ whose all of the derivative at $\{0\}\times M$ vanishes. We denote by $L^p_k([0,T]/0\times M, V)$ the completion of $C^\infty ([0,T]/0\times M, V)$ by the weighted $L^p_k$-sobolev norm (see \cite{Ham1, Kob1}). Furthermore, let $L^p_k([0,T]/0\times M, V)_\#\subseteq L^p_k([0,T]/0\times M, V)$ be a closed subspace defined as 
\begin{align*}
&L^p_k([0,T]/0\times M, V)_\#\coloneqq \\
&\{ (\xi_t)_{t\in [0,T]}\in L^p_k([0,T]/0\times M, V)\mid \xi_t\left.\right|_{\partial M}=0\ \text{for all $t\in[0,T]$}\}.
\end{align*}
We take $p>0$ large enough so that $L^p_2([0,T]/0\times M)\subseteq C^0([0,T]\times M, V)$ and that \cite[the first theorem on page 116]{Ham1} holds. Let $(\widetilde{\xi}_t:M\rightarrow V)_{t\in[0,T]}$ be a smooth 1-parameter family of $V$-valued functions such that $\xi_0=\eta$ and that $\widetilde{\xi}_t\left.\right|_{\partial M}=\eta\left.\right|_{\partial M}$ for all $t\in[0,T]$. We define a map $H:L^p_2([0,T]/0\times M)_\#\rightarrow L^p([0,T]\times M)$ as follows:
\begin{align*}
H(\xi^\#_t)\coloneqq \frac{d(\xi^\#_t+\widetilde{\xi}_t)}{dt}+F(\xi^\#_t+\widetilde{\xi}_t)\ \text{for $\xi^\#_t\in L^p_2([0,T]/0\times M)$},
\end{align*}
where $F$ is defined in Definition \ref{F, G}. The linearization of $H$ at 0, denoted as $H_\ast$, is given by the following:
\begin{align*}
H_\ast(\chi_t)=\frac{d\chi_t}{dt}+\Delta_{g_M}\chi_t+\sum_{j=1}^ra_je^{(v_j,\widetilde{\xi}_t)}(v_j,\chi_t)v_j 
\end{align*}
for all $\chi_t\in L^p_2([0,T]/0\times M)$. Let $\chi_t\in \ker H_\ast$. Then from the regularity theorem for parabolic operators (see \cite{Ham1, Kob1}), $\chi_t$ is smooth on $(0,T]\times M$. Moreover, we have
\begin{align*}
\frac{1}{2}\left(\frac{d}{dt}+\Delta_{g_M}\right)|\chi_t|^2\leq -\sum_{j=1}^ra_je^{(v_j,\widetilde{\xi}_t)}(v_j,\chi_t)^2\leq 0.
\end{align*}
Then from Lemma \ref{max}, $\chi_t=0$ and thus $H_\ast$ is injective. By using \cite[the first theorem on page 116]{Ham1} and the invariance of indices of Fredholm operators under perturbations by compact operators (see \cite{Ham1, Kob1}), we see that $H_\ast$ is bijective. From the implicit function theorem (see \cite{Ham1, Kob1}), we can find a small $0<\epsilon \ll T$ and a $\xi^\#\in L^p_2([0,\epsilon)/0\times M)_\#$ such that $H(\widetilde{\xi}_t+\xi^\#_t)=0$. Then from the regularity theorem for parabolic operators (see \cite{Ham1, Kob1}), $\xi^\#_t$ is smooth on $(0,\epsilon)\times M$.

Thirdly, we show the global existence of a solution to equation (\ref{heat equation}) satisfying the boundary condition. Suppose that for a finite $T>0$, we have a solution $(\xi_t)_{t\in[0,T)}$ satisfying the boundary condition. From Lemma \ref{lemma1} and Lemma \ref{max}, by using the same argument as \cite[Corollary 15]{Don0}, we see that there exists a continuous map $\xi_T: M\rightarrow V$ satisfying the boundary condition $\xi_T\left.\right|_{\partial M}=\eta\left.\right|_{\partial M}$ such that $\lim_{t\rightarrow T}\sup_M|\xi_t-\xi_T|=0$. From Lemma \ref{lemma2} and Lemma \ref{max}, we also see that $\sup_M|\Delta_{g_M}\xi_t|$ is bounded for $t\in [0, T)$. Then from the $L^p$-estimate (see \cite[p.96]{Ham1}) for Laplacian $\Delta_{g_M}$, we see $|\xi_t|_{L^p_2}$ is bounded for any $1<p<\infty$, and thus $\sup_M(|d\xi_t|+|\xi_t|)$ is bounded for $t\in[0,T)$. Then from Lemma \ref{curvature}, we see that $f\coloneqq |d F(\xi_t)|^2$ satisfies
\begin{align*}
\left(\frac{d}{dt}+\Delta_{g_M}\right)f\leq A(f+1)
\end{align*}
for a positive constant $A$. Then from the argument of \cite[pp.120-121, proof of (8.15)]{Kob1} and the existence theorem of the solution of the linear heat equation \cite[pp.116-118]{Ham1}, we see that $\sup_M|d F(\xi_t)|$ is bounded for $t\in[0,T)$. Then by repeating the argument above, we see that $\sup_M(|\xi_t|+|d\xi_t|+|\nabla^2\xi_t|)$ is bounded for $t\in [0, T)$. By induction, we see that $\sup_M|\nabla^kF(\xi_t)|$ and $\sup_M(\sum_{j=0}^k|\nabla^j\xi_t|)$ are bounded for every $k$. From this, we also easily see that $|\nabla^l\frac{d^k\xi_t}{dt^k}|$ is bounded for all $k$ and $l$. Therefore, $(\xi_t)_{t\in [0,T)}$ converges to $\xi_T$ in $C^\infty$. This implies the global existence of a solution of equation (\ref{heat equation}).

Fourthly, for the solution $(\xi_t)_{t\in [0,\infty)}$ of equation (\ref{heat equation}) satisfying the boundary condition (\ref{boundary}), we prove (i) of the latter statement of Theorem \ref{main theorem 1}. Suppose that $\partial M$ is not empty. We first note the uniqueness of the solution of equation (\ref{KW}) satisfying the same boundary condition follows from Lemma \ref{xi-} and the maximum principle (see \cite[p.94]{Don1}). For the rest of the proof of (i) of Theorem \ref{main theorem 1}, we use the argument of \cite{Don1}. Since $F(\xi_t)\left.\right|_{\partial M}=0$ for all $t\in [0,\infty)$ and we have Lemma \ref{lemma2}, by applying \cite[Lemma on p.98]{Don1} to $|F(\xi_t)|^2$, we see that there exists a positive constant $C$ and $\mu$ such that
\begin{align*}
\sup_M|F(\xi_t)|^2\leq Ce^{-\mu t} \ \text{for all $t\in[0,\infty)$}.
\end{align*}
From this, we see that $\sup_M|\xi_t|$ is bounded for $t\in [0,\infty)$. Then by the same argument as in the proof of the global existence of $(\ref{heat equation})$ above, we see that for all $k=0,1,2,\dots,$ $\sup_M(\sum_{j=0}^k|\nabla^j\xi_t|)$ and $\sup_M|\nabla^k F(\xi_t)|$ are bounded for $t\in [0,\infty)$ and thus $|\nabla^l\frac{d^k\xi_t}{dt^k}|$ is bounded for all $k$ and $l$. Therefore we can find a sequence $(t_i)_{i\in \N}$ going to infinity and a smooth $\xi: M\rightarrow V$ such that $\xi_{t_i}\rightarrow \xi$ in $C^\infty$. This $\xi$ solves equation (\ref{KW}) and satisfies the Dirichlet boundary condition. 

Finally, we prove (ii) of Theorem \ref{main theorem 1}. Suppose that the boundary $\partial M$ is empty and that for each $j=1,\dots, r$, $a_j^{-1}(0)$ is a measure $0$ set and $\log a_j$ is integrable. Let $E$ be the functional introduced in \cite{Miy1}. From the assumption, we see that $E$ is bounded below and the estimate of \cite[Lemma 3]{Miy1} holds. By calculation we have
\begin{align*}
\frac{d}{dt} E(\xi_t)=-\int_M|F(\xi_t)|^2.
\end{align*}
In particular, $E(\xi_t)$ is bounded above and thus $|\xi_t|_{L^2}$ is bounded for $t\in[0,\infty)$. Then from \cite[pp.72-73]{LT1} (see \cite[Lemma 6]{Miy1}) and the proof of \cite[Lemma 5]{Miy1}, we see that $\sup_M|\xi_t|$ is bounded for $t\in[0,\infty)$. Then by repeating the argument of the proof of (i) of Theorem \ref{main theorem 1}, we can find a sequence $(t_i)_{i\in \N}$ going to infinity and a smooth $\xi: M\rightarrow V$ such that $\xi_{t_i}\rightarrow \xi$ in $C^\infty$. This $\xi$ solves equation (\ref{KW}). The uniqueness of the solution of equation (\ref{KW}) follows from \cite[Theorem 1]{Miy1}.
\end{proof}
\begin{proof}[Proof of Theorem \ref{main theorem 3}]
For each $j=1,\dots, r$, let $f_j\coloneqq (\xi-\xi^\prime,u_j)$. We set $\varphi\coloneqq \sum_{j=1}^re^{f_j/2}u_j$, $\widetilde{\varphi}\coloneqq \sum_{j=1}^re^{f_j}u_j$. Then we can calculate $\Delta_{g_M}\log|\sum_{j=1}^re^{(\xi-\xi^\prime,u_j)}|=\Delta_{g_M}\log|\varphi|^2$ as follows:
\begin{align*}
\Delta_{g_M}\log|\varphi|^2&=\frac{\Delta_{g_M}|\varphi|^2}{|\varphi|^2}+\frac{|d|\varphi|^2|^2}{|\varphi|^4} \\
&\leq \frac{2(\Delta_{g_M}\varphi,\varphi)}{|\varphi|^2}-\frac{2|d\varphi|^2}{|\varphi|^2}+\frac{4|d\varphi|^2}{|\varphi|^2} \\
&= \frac{2(\Delta_{g_M}\varphi,\varphi)}{|\varphi|^2}+\frac{2|d\varphi|^2}{|\varphi|^2},
\end{align*}
where we have used $d|\varphi|^2=2(d\varphi,\varphi)$ and the inequality $|(d\varphi,\varphi)|^2\leq |d\varphi|^2|\varphi|^2$. Furthermore, $\Delta_{g_M}\varphi$ is calculated as follows: 
\begin{align*}
\Delta_{g_M}\varphi=\sum_{j=1}^r\frac{1}{2}\Delta_{g_M}f_je^{\frac{f_j}{2}}u_j-\sum_{j=1}^r\frac{1}{4}|df_j|^2e^{\frac{f_j}{2}}u_j.
\end{align*}
Therefore, we have
\begin{align*}
(\Delta_{g_M}\varphi,\varphi)&=\sum_{j=1}^r(\frac{1}{2}\Delta_{g_M}f_je^{\frac{f_j}{2}}u_j-\sum_{j=1}^r\frac{1}{4}|df_j|^2e^{\frac{f_j}{2}}u_j,\varphi) \\
&=\frac{1}{2}(\Delta_{g_M}(\xi-\xi^\prime),\widetilde{\varphi})-|d\varphi|^2.
\end{align*}
Then we have the following:
\begin{align*}
\Delta_{g_M}\log|\varphi|^2\leq \frac{(\Delta_{g_M}(\xi-\xi^\prime),\widetilde{\varphi})}{|\varphi|^2}.
\end{align*}
Since, for each $j=1,\dots, r$, the function 
\begin{align*}
((e^{(v_j,\xi)}-e^{(v_j,\xi^\prime)})v_j,\widetilde{\varphi})=(e^{(v_j,\xi)}-e^{(v_j,\xi^\prime)})(e^{(\xi-\xi^\prime,u_j)}-e^{(\xi-\xi^\prime,u_{j+1})})
\end{align*}
is a positive function, we have
\begin{align*}
&\Delta_{g_M}\log|\varphi|^2 \\
\leq &\frac{(\Delta_{g_M}(\xi-\xi^\prime),\widetilde{\varphi})}{|\varphi|^2} \\
\leq &\frac{1}{|\varphi|^2}\{(\Delta_{g_M}\xi+\sum_{j=1}^ra_je^{(v_j,\xi)}v_j-w,\widetilde{\varphi}) 
-(\Delta_{g_M}\xi^\prime+\sum_{j=1}^ra_je^{(v_j,\xi^\prime)}v_j-w, \widetilde{\varphi})\}. \\
\leq &\frac{|\widetilde{\varphi}|}{|\varphi|}\{|\Delta_{g_M}\xi+\sum_{j=1}^ra_je^{(v_j,\xi)}v_j-w| 
+|\Delta_{g_M}\xi^\prime+\sum_{j=1}^ra_je^{(v_j,\xi^\prime)}v_j-w|\}. \\
\leq &|\Delta_{g_M}\xi+\sum_{j=1}^ra_je^{(v_j,\xi)}v_j-w| +|\Delta_{g_M}\xi^\prime+\sum_{j=1}^ra_je^{(v_j,\xi^\prime)}v_j-w|. 
\end{align*}
Then we have the desired claim.
\end{proof}
\begin{proof}[Proof of Theorem \ref{main theorem 4}] Let $\xi$ be a $C^2$-solution of (\ref{KW}). For each $j=1,\dots, r$, let $f_j\coloneqq (v_j,\xi)+\log(a_j)$. We set $\varphi\coloneqq \sum_{j=1}^re^{\frac{1}{2}f_j}u_j, \Psi\coloneqq \sum_{j=1}^re^{f_j}v_j$. By the same calculation as in the proof of Theorem \ref{main theorem 3}, we have
\begin{align*}
\Delta_{g_M}\log|\varphi|^2 \leq &\frac{1}{|\varphi|^2}\sum_{j=1}^r\Delta_{g_M}f_je^{f_j} \\
=&\frac{1}{|\varphi|^2}\sum_{j=1}^r\Delta_{g_M}((v_j,\xi)+\log(a_j))e^{f_j}.
\end{align*}
By assumption, we have
\begin{align*}
&\frac{1}{|\varphi|^2}\sum_{j=1}^r\Delta_{g_M}((v_j,\xi)+\log(a_j))e^{f_j} \\
&\leq \frac{1}{|\varphi|^2}\sum_{j=1}^r (\Delta_{g_M}\xi-w,v_j)e^{f_j} \\
&=-\frac{1}{|\varphi|^2}\sum_{j=1}^r(\Psi, v_j)e^{f_j} \\
&=-\frac{|\Psi|^2}{|\varphi|^2}.
\end{align*}
Then we have the desired claim.
\end{proof}

\noindent
{\bf Acknowledgements.} I am very grateful to Takahiro Aoi for his valuable discussions. I would like to express my gratitude to Ryushi Goto, Yoshinori Hashimoto, and Hisashi Kasuya for their valuable discussions and many supports. I would also like to express my sincere gratitude to Takuro Mochizuki for answering my many questions, and for informing me of many things about harmonic bundles and cyclic Higgs bundles.

\noindent
E-mail address 1: natsuo.m.math@gmail.com \\
\noindent
E-mail address 2: n-miyatake@imi.kyushu-u.ac.jp

\noindent
Institute of Mathematics for Industry, Kyushu University 744 Motooka, Fukuoka
819-0395, Japan

\begin{thebibliography}{99}
\bibitem{Bar1} D. Baraglia, {\it Cyclic Higgs bundles and the affine Toda equations}, Geom. Dedicata 174 (2015), 25–42.
\bibitem{DL1} S. Dai and Q. Li, {\it On cyclic Higgs bundles}, Math. Ann. 376 (2020), no. 3-4, 1225–1260.
\bibitem{Don0} S.K. Donaldson, {\it Anti self‐dual Yang‐Mills connections over complex algebraic surfaces and stable vector bundles}, Proceedings of the London Mathematical Society 3.1 (1985): 1-26.
\bibitem{Don1} S. K. Donaldson, {\it Boundary value problems for Yang-Mills fields}, Journal of geometry and physics 8.1-4 (1992): 89-122.
\bibitem{GL1}M. A. Guest and C.-S. Lin, {\it Nonlinear PDE aspects of the $tt^\ast$ equations of Cecotti and Vafa}, J. Reine Angew. Math. 689 (2014), 1–32.
\bibitem{Ham1} R. S. Hamilton, {\it Harmonic maps of manifolds with boundary}, Vol. 471. Springer, 2006.
\bibitem{Kob1} S. Kobayashi, {\it Differential geometry of complex vector bundles}, Vol. 793. Princeton University Press, 2014.
\bibitem{LM1}Q. Li and T. Mochizuki. {\it Complete solutions of Toda equations and cyclic Higgs bundles over non-compact surfaces}, arXiv:2010.05401 (2020).
\bibitem{LM2} Q. Li and T. Mochizuki, {\it Isolated singularities of Toda equations and cyclic Higgs bundles}, arXiv:2010.06129 (2020).
\bibitem{LT1} M. L\"ubke and A. Teleman, {\it The Kobayashi-Hitchin correspondence}, World Scientific Publishing Co., Inc., River Edge, NJ, 1995. x+254 pp. ISBN: 981-02-2168-1.
\bibitem{Miy1} N. Miyatake, {\it Generalized Kazdan-Warner equations associated with a linear action of a torus on a complex vector space}, Geom Dedicata 214, 651–669 (2021).
\bibitem{Miy3} N. Miyatake, {\it Generalized Kazdan-Warner equations on foliated manifolds}, arXiv:2204.01253 (2022).
\bibitem{Miy4} N. Miyatake, {\it Restriction of Donaldson's functional to diagonal metrics on Higgs bundles with non-holomorphic Higgs fields}, arXiv:2301.01485 (2023).
\bibitem{Ran1} T. Ransford, {\it Potential theory in the complex plane}, No. 28. Cambridge university press, 1995.
\bibitem{Sim1} C.T. Simpson, {\it Constructing variations of Hodge structure using Yang-Mills theory and applications to uniformization,} J. Amer. Math. Soc. 1 (1988), no. 4, 867–918. 
\end{thebibliography}
\end{document}